\documentclass[twoside]{amsart}
\usepackage{latexsym}
\usepackage{amssymb,amsmath,amsopn}
\usepackage[dvips]{graphicx}   
\usepackage{color,epsfig}      
\usepackage{verbatim,hyperref}      

               {\begin{list}{}{\leftmargin#1\rightmargin#2}\item{}}%
               {\end{list}}

\newtheorem{thm}{Theorem}
\newtheorem{lem}{Lemma}
\newtheorem{prop}{Proposition}
\theoremstyle{definition}

\newtheorem{rem}{Remark}
\newtheorem{ques}{Question}

\newtheorem{prob}{Problem}

\renewcommand{\Re}{\mathbb R}

\newcommand{\BB}{\mathbf B}

\renewcommand{\S}{\mathbb{S}}

\def\eea{\end{eqnarray}}

\DeclareMathOperator{\bd}{bd}

\DeclareMathOperator{\conv}{conv}

\definecolor{green}{rgb}{0.0, 0.75, 0.25}

\begin{document}

\title[Mono-monostatic convex bodies]{A characterization of the symmetry groups of mono-monostatic convex bodies}

\author[G. Domokos]{G\'abor Domokos}
\author[Z. L\'angi]{Zsolt L\'angi}
\author[P. V\'arkonyi]{P\'eter L. V\'arkonyi}

\address{MTA-BME Morphodynamics Research Group and Department of Mechanics, Materials \& Structures, Budapest University of Technology, Műegyetem rakpart 1-3., Budapest, Hungary, 1111} 
\email{domokos@iit.bme.hu}
\address{MTA-BME Morphodynamics Research Group and Department of Geometry, Budapest University of Technology, Egry J\'{o}zsef utca 1., Budapest, Hungary, 1111} 
\email{zlangi@math.bme.hu}
\address{Department of Mechanics, Materials \& Structures, Budapest University of Technology, Műegyetem rakpart 1-3., Budapest, Hungary, 1111} 
\email{varkonyi.peter@epk.bme.hu}

\thanks{The research reported in this paper was supported by the
Thematic Excellence Program of the Ministry for Innovation and Technology in Hungary TKP2021-BME-NVA at the Budapest University of Technology and
the NKFIH grant K134199.}

\subjclass[2010]{52B10, 52B15, 74G05}
\keywords{monostable, mono-monostatic, symmetry}

\begin{abstract}
Answering a question of Conway and Guy in a 1968 paper, L\'angi in 2021 proved the existence of a monostable polyhedron with $n$-fold rotational symmetry for any $n \geq 3$, and arbitrarily close to a Euclidean ball. In this paper we strengthen this result by characterizing the possible symmetry groups of all mono-monostatic smooth convex bodies and convex polyhedra. Our result also answers a stronger version of the question of Conway and Guy, asked in the above paper of L\'angi.
\end{abstract}

\maketitle

\section{Introduction}

\subsection{Monostatic shapes}
Not only was the general study of the relationship between geometric shape and the number of static balance points initiated by Archimedes \cite{archimedes}, his design of ship hulls immediately highlighted the special role of \emph{monostatic} shapes, i.e. bodies having just one stable or one unstable static balance position. In the former case monostatic objects are also referred to as \emph{monostable} while in the latter case as \emph{mono-unstable}.  The same idea was re-visited in 1966 by Conway and Guy \cite{conway}
who conjectured that there is no homogeneous tetrahedron which can stand in rest only on one of its faces when placed on a horizontal plane (i.e. homogeneous, monostable tetrahedra do not exist), but there is a homogeneous convex polyhedron with the same property (i.e. homogeneous, monostable polyhedra do exist).  The full proof of the first conjecture was published by Dawson \cite{Dawson} who credited Conway with the idea.  The second conjecture was resolved in \cite{goldberg} by an explicit monostable polyhedral construction with $D_2$-symmetry, which has been dubbed the Conway-Guy polyhedron.  In the same paper Guy presented some problems regarding monostable polyhedra, stating that three of them are due to Conway. These three questions (which we list below) appear also in the problem collection of Croft, Falconer and Guy \cite{CFG} as Problem B12, as well as in the collection of Klamkin \cite{Klamkin}. Problem~\ref{prob1} and  some other problems for monostable polyhedra appear in a 1968 collection of geometry problems of Shephard \cite{Shephard}, who described these objects as `a remarkable class of convex polyhedra' whose properties `it would probably be very rewarding and interesting to make a study of'.

\begin{prob}\label{prob1}
Can a monostable polyhedron in the Euclidean $3$-space $\Re^3$ have an $n$-fold axis of symmetry for $n > 2$?
\end{prob}

Before the next problem, recall that the \emph{girth} of a convex body in $\Re^3$ is the minimum perimeter of an orthogonal projection of the body onto a plane \cite{dumitrescu}.

\begin{prob}\label{prob2}
What is the smallest possible ratio of diameter to girth for a monostable polyhedron?
\end{prob}

\begin{prob}\label{prob3}
What is the set of convex bodies uniformly approximable by monostable polyhedra, and does this contain the sphere?
\end{prob}

The second-named author of the current paper answered in \cite{langi} the first two questions by showing Theorem~\ref{thm:monostable}, where $d_H(\cdot, \cdot)$ and $\BB^3$ denotes the Hausdorff distance of convex bodies, and the closed unit ball in $\Re^3$ centered at the origin $o$, respectively. 

\begin{thm}[L\'angi, 2021]\label{thm:monostable}
For any $n \geq 3$, $n \in \mathbb{Z}$ and $\varepsilon > 0$ there is a homogeneous monostable polyhedron $P$ such that $P$ has an $n$-fold rotational symmetry and
$d_H(P, \BB^3) < \varepsilon$.
\end{thm}

We remark that among all convex bodies in $\Re^3$, the ratio of diameter to girth is minimal for the Euclidean ball (see, e.g. \cite{langi}), hence Theorem \ref{thm:monostable} implies that the answer to Problem \ref{prob2} is $\frac{1}{\pi}$.

\subsection{Mono-monostatic shapes}
{While the class of monostatic polyhedra already appears to have intriguing geometric properties, the study of an even more elusive class of convex bodies was initiated by Arnold \cite{lunch} who conjectured that smooth, homogeneous, convex bodies with just two static balance points (one stable, one unstable) may exist. His conjecture was proven in \cite{gomboc} where this class of bodies was dubbed \emph{mono-monostatic} and  an explicit construction with $D_2$-symmetry was presented. On the other hand, the general symmetry properties of smooth mono-monostatic shapes have not been explored.
Little is known about the polyhedral case either. In \cite{langi}, the author asked the following question, as a stronger version  of Problem~\ref{prob1}.

\begin{ques}[L\'angi, 2021] \label{ques1}
What are the positive integers $n \geq 2$ such that there is a mono-monostatic convex polyhedron with an $n$-fold axis of symmetry?
\end{ques}

In this paper we answer a more general question which goes beyond Question \ref{ques1} and also encompasses the smooth case: our aim is to characterize all
the possible symmetry groups of mono-monostatic bodies, both smooth and polyhedral.

\subsection{Basic concepts}\label{ss:con}
Before formally  stating our main result, we recall some basic concepts from the theory of equilibrium points.

Let $K \subset \Re^3$ be a smooth convex body, i.e. having $C^1$-class boundary, and let the center of mass of $K$ be $p$. Let $\delta_K : \bd(K) \to \Re$ be the Euclidean distance function measured from the point $p$, where $\bd(K)$ denotes the boundary of $K$. The critical points of $\delta_K$ are called \emph{equilibrium points} of $K$. To avoid degeneracy, it is usually assumed that $\delta_K$ is a Morse function; i.e. it has finitely many critical points, $\bd(K)$ is twice continuously differentiable at least in a neighborhood of each critical point, and at each such point the Hessian of $\delta_K$ is nondegenerate. Depending on the number of negative eigenvalues of the Hessian, we distinguish between \emph{stable}, \emph{unstable} and \emph{saddle-type} equilibrium points, corresponding to the local minima, maxima and saddle points of $\delta_K$, respectively. The Poincar\'e-Hopf Theorem implies that under these conditions, the numbers $S$, $U$ and $H$ of the stable, unstable and saddle points of $K$, respectively, satisfy the equation $S-H+U=2$.
For any $S, U \geq 1$ we define the set $(S,U)_c$ as the family of smooth convex bodies $K$ having $S$ stable and $U$ unstable equilibrium points, where $K$ has no degenerate equilibrium point, and at each such point $\bd(K)$ has a positive Gaussian curvature. 

We define the class $(S,U)_p$ analogously for convex polyhedra. More specifically, let $P$ be a convex polyhedron in $\Re^3$ with $p$ as its center of mass. A point $q \in \bd(P)$ is called an equilibrium point of $P$ if the plane through $q$ with normal vector $q-p$ supports $P$. Note that in this case there is a unique vertex, edge or face of $P$ that contains $q$ in its relative interior, where by the relative interior of a vertex we mean the vertex itself. Let $F$ denote this vertex, edge or face, and let $H$ be the supporting plane of $P$ through $q$ that is perpendicular to $q-p$. Observe that $F \subset P \cap H$. We say that $q$ is \emph{nondegenerate} if $F= P \cap H$, and $P$ is nondegenerate if all its equilibrium points are nondegenerate. A nondegenerate equilibrium point $q$ is called a \emph{stable, saddle-type} or \emph{unstable} point of $P$ if the dimension of $F$ is $2,1$ or $0$, respectively \cite{DKLRV}. Again, we may apply the Poincar\'e-Hopf Theorem and obtain that if $P$ is nondegenerate, then the numbers $S$, $U$ and $H$ of the stable, unstable and saddle points of $K$, respectively, satisfy the equation $S-H+U=2$. Thus, for any $S,U \geq 1$, we define the class $(S,U)_p$ as the family of nondegenerate convex polyhedra with $S$ stable and $U$ unstable points (see also \cite{DKLRV} or \cite{langi}). We remark that by \cite{gomboc} and \cite{langi}, the classes $(1,1)_c$ and $(1,1)_p$ are not empty.

Finally, we call a convex body $K \subset \Re^3$ \emph{centered} if its center of mass is the origin $o$.

\subsection{Main result}
Using the concepts of Subsection \ref{ss:con}, we formulate our main result about the possible symmetry groups of mono-monostatic bodies as
\begin{thm}\label{thm:main}
Let $\mathcal{F} = (1,1)_c$ or $\mathcal{F} =(1,1)_p$. Then
\begin{itemize}
\item[(i)] the symmetry group of any element of $\mathcal{F}$ is a discrete $2$-dimensional point group and
\item[(ii)] for any discrete $2$-dimensional point group $G$ and any $\varepsilon > 0$, there is an element $K \in \mathcal{F}$ whose symmetry group is $G$ and which satisfies $d_H(K,\BB^3) < \varepsilon$.
\end{itemize}
\end{thm}

\begin{rem}\label{rem:listofgroups}
We remark that any discrete $2$-dimensional point group is either a dihedral group $D_n$ for some $n \geq 1$, or a rotation group $C_n$ for some $n \geq 1$; here the latter group is isomorphic to the $n$-element cyclic group $\mathbb{Z}_n$ (see e.g. \cite{lft}).
\end{rem}

The proof of Theorem~\ref{thm:main} has two parts. In the first part, contained in Section~\ref{sec:construction}, we show that for any $\varepsilon > 0$ and $n \geq 2$, the family $(1,1)_c$ contains an element $K$ with symmetry group $D_n$ and satisfying $d_H(K,\BB^3) < \varepsilon$. In the second part we show how the existence of such a body implies Theorem~\ref{thm:main}.

\section{The construction of a smooth mono-monostatic convex body with $D_n$ symmetry}\label{sec:construction}

In this section we fix $n \geq 2$. Our construction roughly follows the construction in \cite{gomboc} with suitable modifications to obtain a mono-monostatic convex body with $D_n$ as its symmetry group, and belonging to class $\mathcal{C}^2_+$, i.e. having $C^2$-class boundary and positive Gaussian curvature at every boundary point (cf. e.g. \cite{Schneider}).

To achieve our goal, we define a two-parameter family $\mathcal{F}_n$ of star-shaped bodies $K_n(c,d)$ such that for any choice of the parameters
\begin{itemize}
\item[(A)] the symmetry group of $K_n(c,d)$ is $D_n$;
\item[(B)] $K_n(c,d)$ has $C^2$-class boundary;
\item[(C)] $K_n(c,d)$ has exactly one stable and one unstable point with respect to $o$.
\end{itemize}    
In the next part of the construction, we show that for any $\varepsilon > 0$, there is a suitable choice of the parameters $c,d$ such that
\begin{itemize}
\item[(D)] $K_n(c,d)$ has positive Gaussian curvature at every boundary point, implying also that it is convex;
\item[(E)] the center of mass of $K_n(c,d)$ is $o$;
\item[(F)] $d_H(K_n(c,d), \BB^3) < \varepsilon$.
\end{itemize}

We note that a convex body satisfying the properties (A)-(F) also satisfies the conditions in (ii) of Theorem~\ref{thm:main} in the case $\mathcal{F} = (1,1)_c$ and $G = D_n$.

We describe the construction of the family $\mathcal{F}_n$ in Subsection~\ref{subsec:construction}. Then we prove that the properties (A)-(C) are satisfied for all members of this family in Subsection \ref{subsec:A-C}. In Subsection~\ref{subsec:properties} we show that a suitable member of this family also satisfies the properties (D)-(F).

\begin{figure}
\centering
\includegraphics[width=6cm]{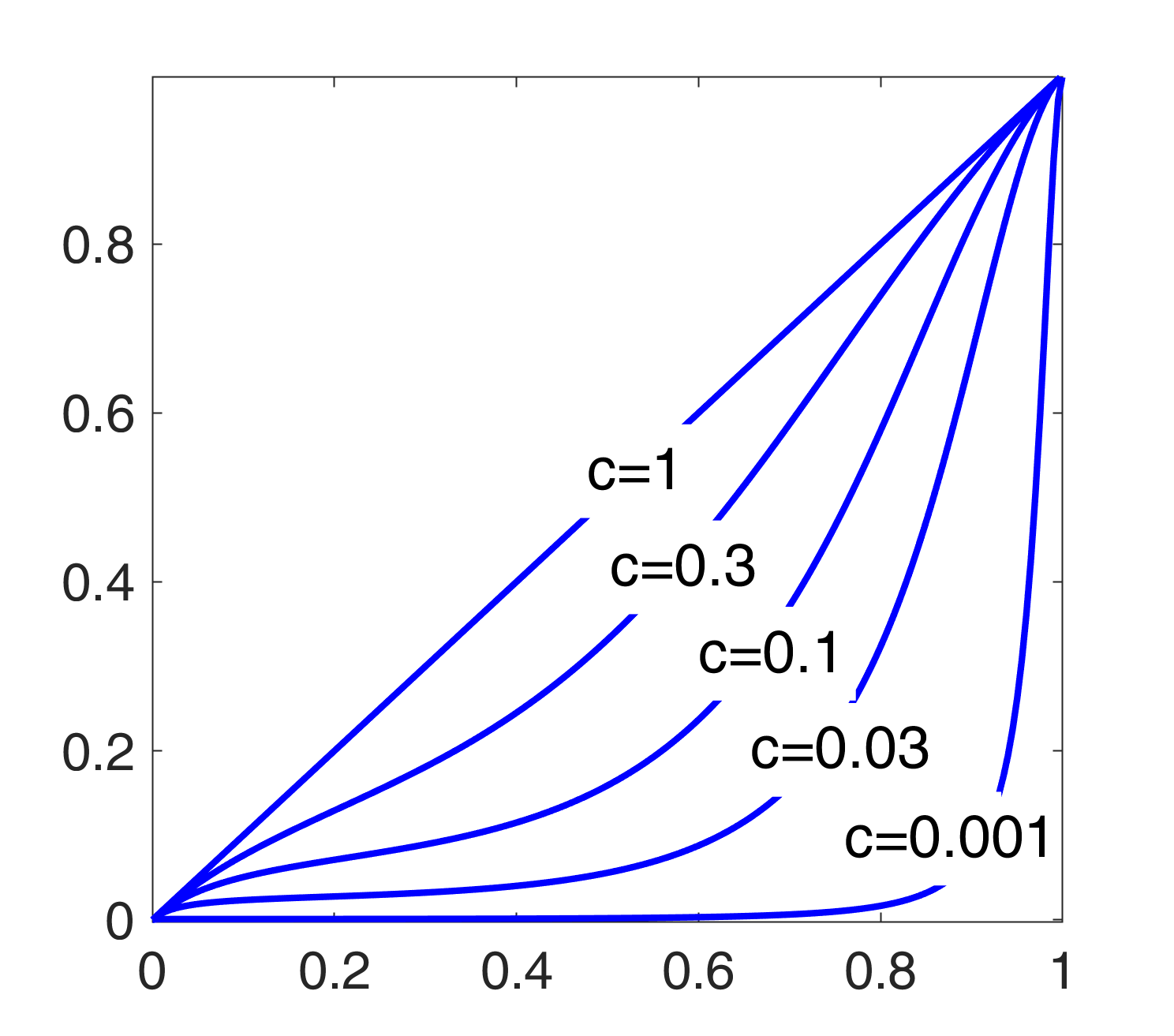}
\caption{Illustration of the function $F_c(x)$.}
\label{fig:Fc}
\end{figure}
                                              
\subsection{Construction of $\mathcal{F}_n$}\label{subsec:construction}

Let us define the function $F_c : [0,1] \to \Re$, 
\begin{equation}\label{eq:defnF}
F_c(x) = \frac{cx^2 + (1-c)(1-x)^2 \cdot \frac{cx}{c+x}}{cx+(1-c)(1-x)^2},
\end{equation}
where $0 < c \leq 1$ (see Fig. \ref{fig:Fc}).

In the next lemma, we collect some properties of $F_c$.
\begin{lem}\label{lem:Fproperties}
For any $c \in (0,1)$, 
\begin{enumerate}
\item[(a)] $F_c$ is smooth.
\item[(b)] $F_c$ is strictly increasing on its domain.
\item[(c)] $F_c(0) = 0$ and $F_c(1) = 1$.
\item[(d)] $(F_c)'_{+}(0)=(F_c)'_{-}(1)=1$.
\end{enumerate}
Furthermore,
\begin{enumerate}
\item[(e)] $F_1(x) = x$ for all $x \in [0,1]$.
\item[(f)] 
$\lim_{c \to 0+0} F_c(x) = 0$ for all $x \in [0,1)$.
\end{enumerate}
\end{lem}

\begin{proof}
To show (a), it is sufficient to show that the denominator of the fraction in (\ref{eq:defnF}) is not zero.
But this fact follows from the observation that the denominator is clearly nonnegative for any $x,c \in [0,1]$, and it is zero if and only if $xc=0$ and $(1-x)(1-c)=0$ are satisfied simultaneously. The solution for this system of equations is $x=0$ and $c=1$, or $x=1$ and $c=0$, neither of which satisfies the conditions in the lemma.

To prove (b), we show that $F'_c(x) > 0$ for any $x \in (0,1)$ and $c \in (0,1)$, and remark that for $c=1$ (b) follows from (e) in a straightforward way.
Note that $F_c(x)$ can be recast as a weighted average of two functions
\[
F_c(x)=w(x)F_1(x)+(1-w(x))F_2(x)
\]
with $F_1(x)=x$, $F_2=(c^{-1}+x^{-1})^{-1}$, and
\[w(x)=\left(1+\frac{(1-c)(1-x^2)}{cx} \right)^{-1}.
\]
Then, we have
\[
F_c'(x)=w(x)F_1'(x)+(1-w(x))F_2'(x)+w'(x)(F_1(x)-F_2(x))
\] 
It is straightforward to show that the terms $w(x)$, $1-w(x)$, $F_1'(x)$, $F_2'(x)$, $F_1(x)-F_2(x)$ are positive, which proves (b).

The statements in (c), (d), (e) and (f) can be obtained by substituting into $F_c$ and its derivative, and in case of (f), also by using the continuity of elementary functions.
\end{proof}

We remark that by Lemma~\ref{lem:Fproperties}, the range of $F_c$ is $[0,1]$ for all $c \in (0,1)$.

In the next step, we define $f_c : \left[ - \frac{\pi}{2}, + \frac{\pi}{2} \right] \to \left[ - \frac{\pi}{2}, + \frac{\pi}{2} \right]$ by 
\begin{equation}\label{eq:f}
f_c(\theta) = \pi F_c\left( \frac{\theta}{\pi} + \frac{1}{2} \right) - \frac{\pi}{2}.
\end{equation}
Note that for any $0 < c < 1$, $f_c$ is a linear image of $F_c$.

Now, set $g_c(\theta) = - f_c(-\theta)$. Clearly, the domain and the range of $g_c$ is $\left[ -\frac{\pi}{2}, \frac{\pi}{2} \right]$. In Remark~\ref{rem:elemprop} we summarize the elementary properties of $f_c$ and $g_c$.

\begin{rem}\label{rem:elemprop}
For any $0 < c \leq 1$, we have  
$f_c \left( \frac{\pi}{2} \right)=g_c \left( \frac{\pi}{2} \right)=\pi/2$, 
$f_c \left( -\frac{\pi}{2} \right)=g_c \left( -\frac{\pi}{2} \right)=-\pi/2$, 
$(f_c)'_{\mp}\left( \pm \frac{\pi}{2} \right)=(g_c)'_{\mp} \left( \pm \frac{\pi}{2} \right) =1$, $(f_c)''_{--}\left( \frac{\pi}{2} \right) = -\frac{2(1-c)}{\pi c (1+c)}, (f_c)''_{++}\left( -\frac{\pi}{2} \right) = - \frac{2}{\pi c}$, $(g_c)''_{--}\left( \frac{\pi}{2} \right) = \frac{2}{\pi c}, (g_c)''_{++}\left(-\frac{\pi}{2} \right) = \frac{2(1-c)}{\pi c (1+c)}$.
\end{rem}

For any $0 < c \leq 1$, $0 \leq \theta \leq 2\pi$ and $\varphi \in \Re$, we set
\begin{equation}\label{eq:a}
a_c(\theta, \varphi) = \frac{ \cos^2 \frac{n \varphi}{2} \cos^2 f_c(\theta)}{ \cos^2 \frac{n \varphi}{2} \cos^2 f_c(\theta) + \sin^2 \frac{n \varphi}{2} \cos^2 g_c(\theta) },
\end{equation}
as illustrated by Fig. \ref{fig:ac}. Note that if $\cos \frac{n\varphi}{2} \neq 0$ (i.e. if $\varphi$ is not an integer multiple of $\frac{2\pi}{n}$), then
$a_c(\theta, \varphi) = \frac{1}{1 + \tan^2 \frac{n\varphi}{2} \frac{\cos^2 f_c(\theta)}{\cos^2 g_c(\theta)}}$.
\begin{figure}
\centering
\includegraphics[width=6cm]{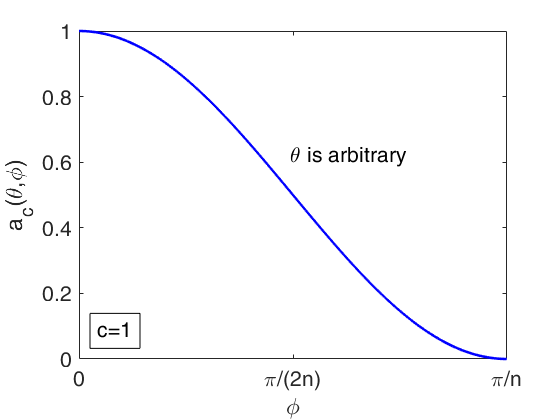}
\includegraphics[width=6cm]{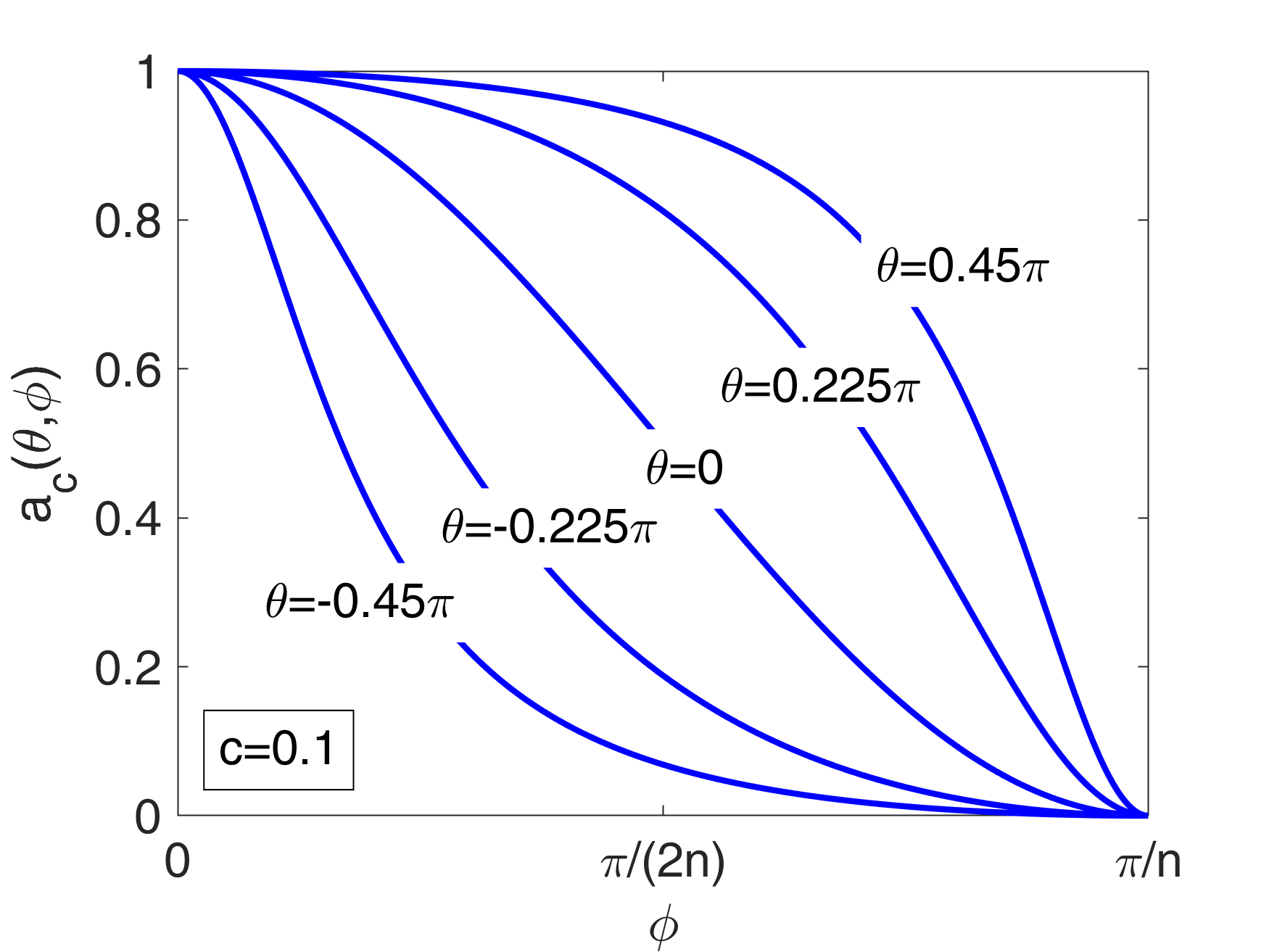}
\caption{Diagram of $a_c(\theta,\varphi)$ versus $\varphi$ for two different values of $c$. If $c=1$ then $a_c=\cos(n\varphi)$ whereas for other values of $c$ and $\theta\neq 0$, it becomes somewhat distorted.}
\label{fig:ac}
\end{figure}
\begin{figure}
\centering
\includegraphics[width=8cm]{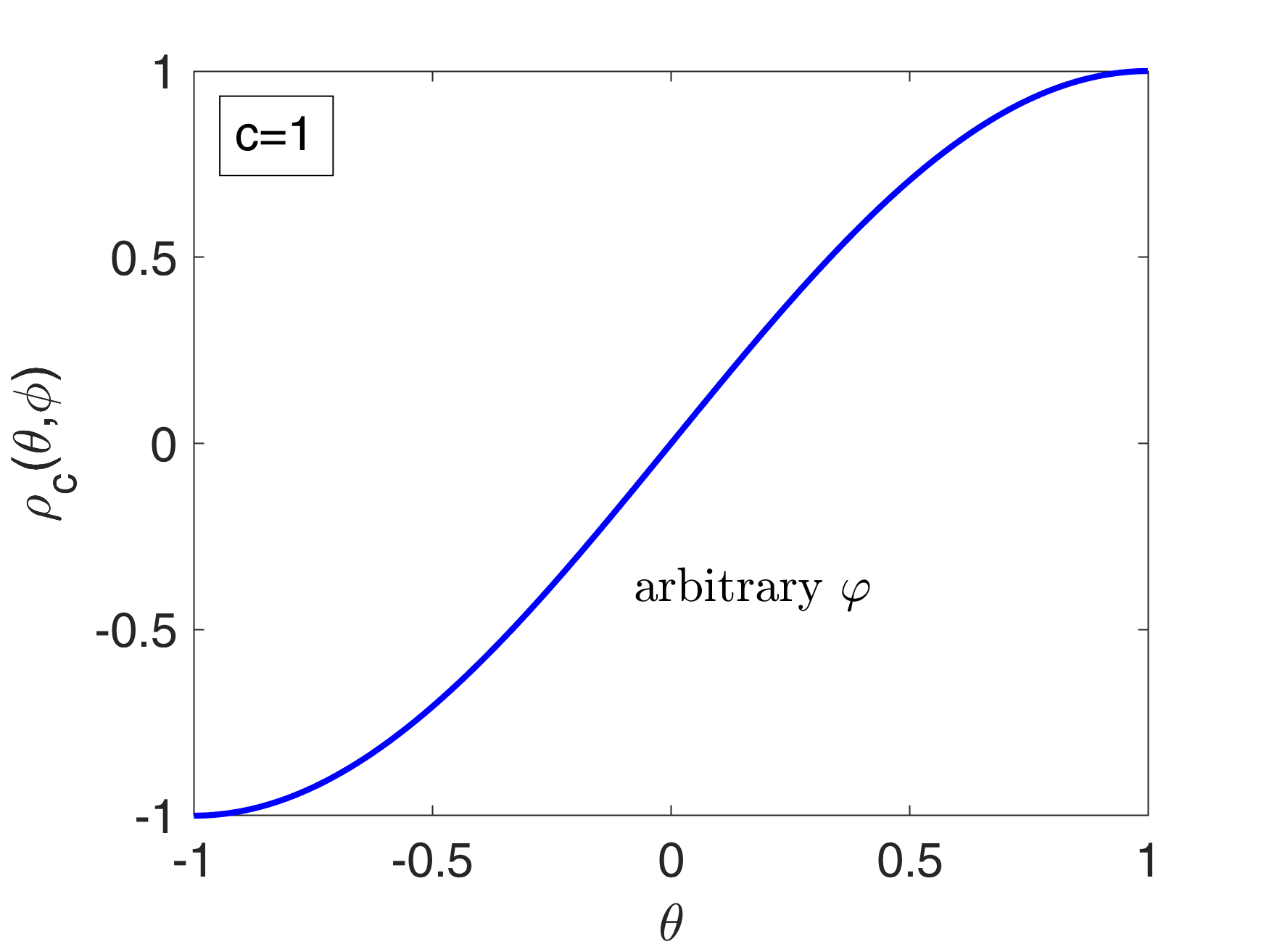}
\includegraphics[width=8cm]{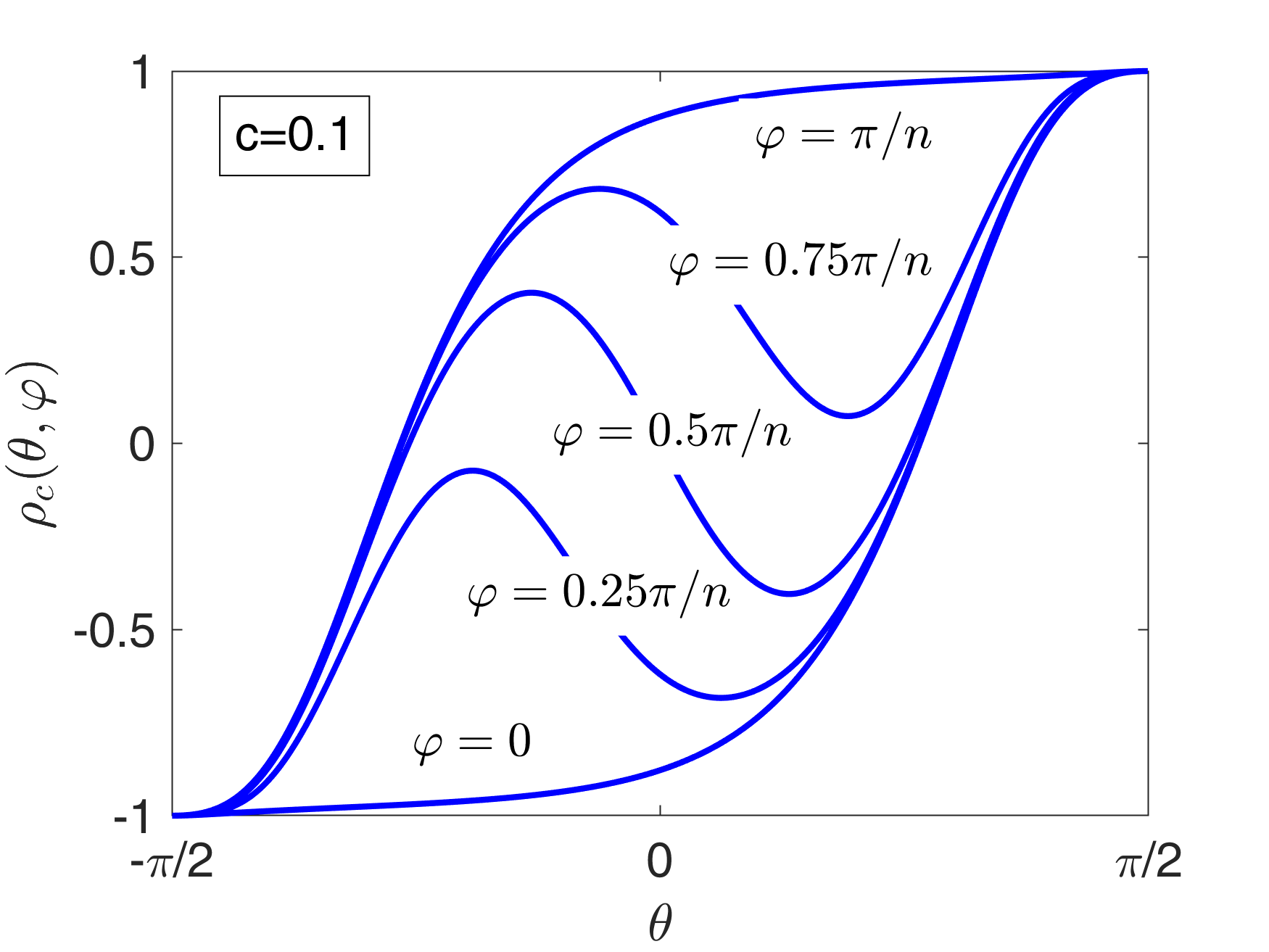}
\includegraphics[width=8cm]{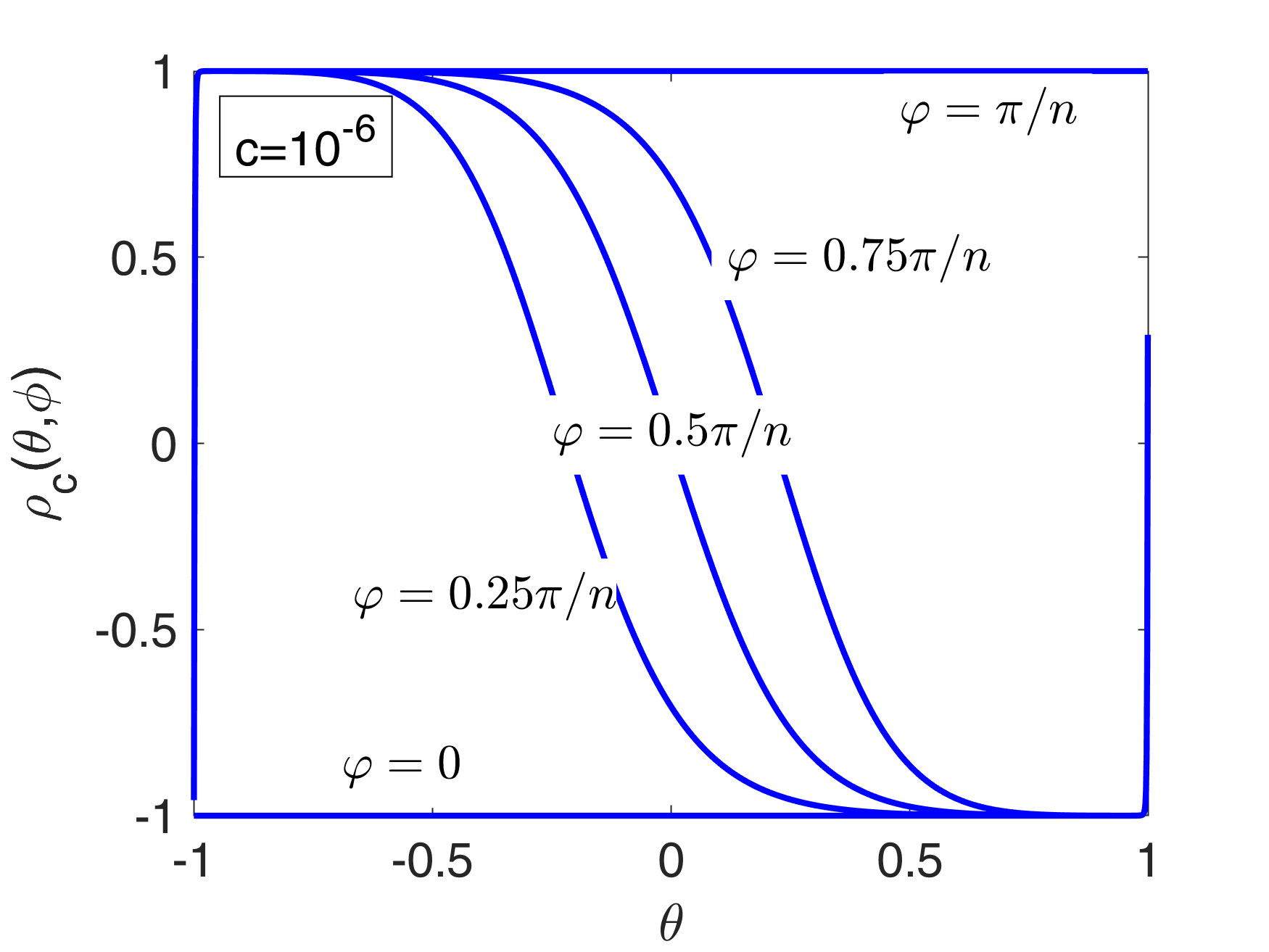}
\caption{Diagram of $\rho_c(\theta,\varphi)$ versus $\theta$ for $c=1,0.1,10^{-6}$. Note that $\rho_c(\theta,0)=f_c(\theta)$ and $\rho_c(\theta,\pi/n)=g_c(\theta)$}
\label{fig:rhoc}
\end{figure}
\begin{figure}
\centering
\includegraphics[width=10cm]{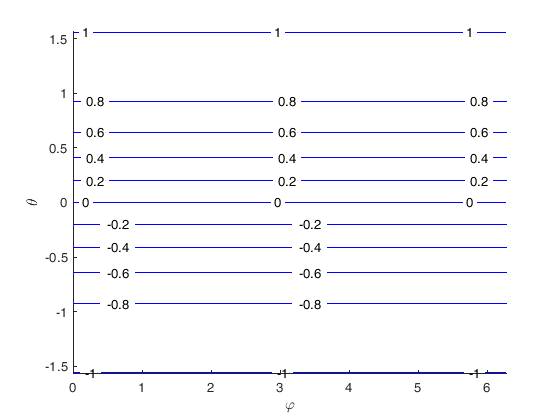}
\includegraphics[width=10cm]{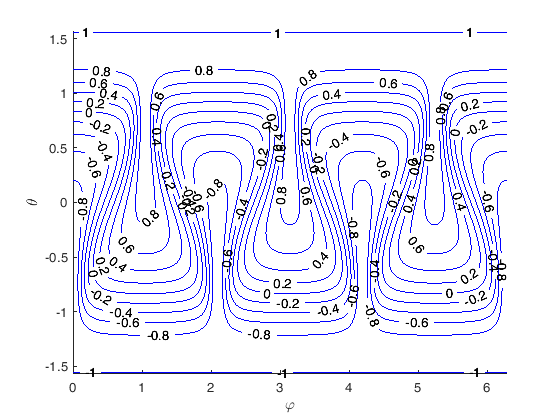}
\caption{Contour plot of $\rho_0(\theta,\varphi)$ for $n=3$ and $c=1$ (top) or $c=0.1$ (bottom). In the first case, $\rho_0$ has the same sign as $\theta$, which yields $H(1,d)>0$ in Eq. \eqref{eq:H} below. In the second case, note large areas with positive values of $\rho_0$ for $\theta<0$ as well as areas with negative values for $\theta>0$. These values explain $H(c,d)<0$ for sufficiently small values of $c$ and $d$.}
\label{fig:rho0contour}
\end{figure}
Using polar coordinates on the unit sphere $\S^2$, we define the function $\rho_c : \S^2 \to \Re$ by
\begin{equation}\label{eq:rho}
\rho_c(u) =  a_c(\theta, \varphi) \sin f_c(\theta) + (1 - a_c(\theta, \varphi)) \sin g_c(\theta),
\end{equation}
for any $u = (\cos \theta \cos \varphi, \cos \theta \sin \varphi, \sin \theta ) \in \S^2$ 
(see Fig. \ref{fig:rhoc} and Fig. \ref{fig:rho0contour}) and set
\begin{equation}\label{eq:R}
R_{c,d}(u) = 1 + d \rho_c(u).
\end{equation}
Observe that by Remark~\ref{rem:elemprop}, both $\rho_c$ and $R_{c,d}$ are well-defined at the two poles, with parameters $\theta = \pm \frac{\pi}{2}$.

Finally, we define the star-shaped set $K_n(c,d)$ (see Fig. \ref{fig:K}) as
\[
K_n(c,d) = \left\{ \lambda u : u \in \S^2, 0 \leq \lambda \leq R_{c,d}(u) \right\},
\]
and note that by definition, for any value of the parameters we have $0 \leq a_c(\theta, \varphi) \leq 1$, and thus, $\rho_c(u)$ is a convex combination of $f_c(\theta)$ and $g_c(\theta)$. This, by Lemma~\ref{lem:Fproperties}, implies that $-1 \leq \rho_c(u) \leq 1$ and $1-d \leq R_{c,d}(u) \leq 1+d$. Thus, $K_n(c,d)$ exists for all $0 \leq d < 1$. 

\begin{figure}
\centering
\includegraphics[width=\columnwidth]{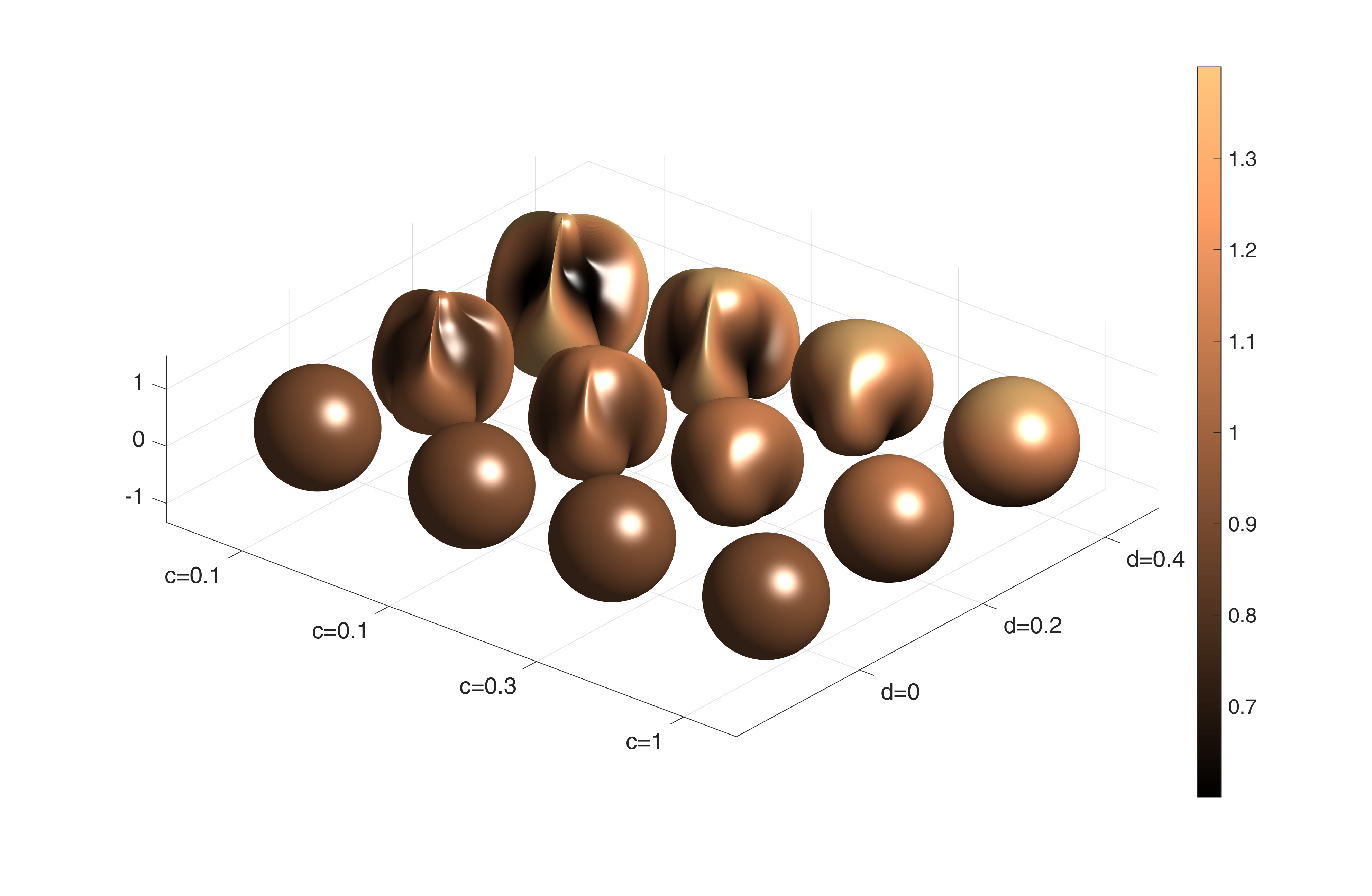}
\caption{The set $K_n(c,d)$ for some values of $c$ and $d$.}
\label{fig:K}
\end{figure}

\subsection{Proof of properties (A)-(C)}\label{subsec:A-C}
In this subsection we prove that the properties described in (A)-(C) are satisfied for $K_n(c,d)$ for all $0 < d < 1$ and $0 < c \leq 1$.

First, observe that by its definition, the symmetry group of $K_n(c,d)$ is the group $D_n$, and thus, (A) is satisfied.
To show (B), first we prove Lemma~\ref{lem:limits}, which plays a crucial role in the proof of (B).

\begin{lem}\label{lem:limits}
Let $\theta_0 \in \left\{ \frac{\pi}{2}-0, -\frac{\pi}{2}+0 \right\}$. Then, as $\theta \to \theta_0$, we have
\[
\frac{(\rho_c)'_{\theta}}{\cos \theta} \to 1, \hbox{ and } \frac{(\rho_c)'_{\varphi}}{\cos^2 \theta}, \frac{(\rho_c)''_{\theta \varphi}}{\cos \theta}, \frac{(\rho_c)'_{\varphi \varphi}}{\cos^2 \theta} \to 0.
\]
Furthermore,
\[
\lim_{\theta \to \frac{\pi}{2}-0} (\rho_c)''_{\theta \theta} = - \lim_{\theta \to -\frac{\pi}{2}+0} (\rho_c)''_{\theta \theta} = -1.
\]
\end{lem}

\begin{proof}
We prove Lemma~\ref{lem:limits} only for $\theta_0 = \frac{\pi}{2}-0$, since if $\theta_0 = -\frac{\pi}{2}+0$, then a similar consideration can be applied.

First, observe that by definition, $\rho_c (\theta, \varphi) = a_c(\theta, \varphi) (\sin f_c - \sin g_c) + \sin g_c$.
Furthermore, $g_c$ is independent of $\varphi$, and by Remark~\ref{rem:elemprop} an elementary computation shows that if $\theta \to \theta_0 \left( =\frac{\pi}{2}-0 \right)$, then $\frac{ (\sin g_c)'}{\cos \theta} \to 1$ and $(\sin g_c)'' \to -1$. Thus, letting $\Psi(\theta,\varphi) = \rho_c (\theta, \varphi) - \sin g_c = a_c ( \theta, \varphi) (\sin f_c - \sin g_c)$, it is sufficient to show that as $\theta \to \theta_0$, we have
\[
\frac{\Psi'_{\theta}}{\cos \theta}, \frac{\Psi'_{\varphi}}{\cos^2 \theta}, \frac{\Psi''_{\theta \varphi}}{\cos \theta}, \frac{\Psi'_{\varphi \varphi}}{\cos^2 \theta}, \Psi''_{\theta \theta} \to 0.
\]
Since the computations to prove these limits are fairly tedious and follow similar ideas, we present the proof of only the first one.

Applying the elementary differentiation rules, we obtain
\[
(a_c)'_{\theta}(\theta,\varphi) = \frac{2 \cos^2 \frac{n \varphi}{2} \sin^2 \frac{n \varphi}{2} \left( \tan g_c \cdot (g_c)'- \tan f_c \cdot (f_c)' \right) }{\left( \cos^2 \frac{n \varphi}{2} \cdot \frac{\cos f_c}{\cos g_c} + \sin^2 \frac{n \varphi}{2} \cdot \frac{\cos g_c}{\cos f_c}  \right)^2} .
\] 
By L'Hospital's rule, we have $\lim_{\theta \to \theta_0} \frac{\cos f_c}{\cos g_c} = \lim_{\theta \to \theta_0} \frac{\cos f_c}{\cos \frac{f_c+g_c}{2}} = 1$.
Thus, the denominator of $(a_c)'_{\theta}$ tends to $1$ as $\theta \to \theta_0$. On the other hand, since $\sin f_c - \sin g_c= 2 \sin \frac{f_c-g_c}{2} \cos \frac{f_c+g_c}{2}$ and $\frac{\sin \frac{f_c-g_c}{2}}{\cos \theta} \to 0$ as $\theta \to \theta_0$, the fact that
\[
\left( \tan g_c \cdot (g_c)'- \tan f_c \cdot (f_c)' \right) (\sin f_c - \sin g_c) =
\]
\[
2 \sin \frac{f_c-g_c}{2} \left( \sin g_c \cdot \frac{\cos \frac{f_c+g_c}{2}}{ \cos g_c}\cdot g_c' - \sin f_c \cdot \frac{\cos \frac{f_c+g_c}{2}}{ \cos f_c}\cdot f_c' \right)
\]
yields that $\frac{(a_c)'_{\theta} (\sin f_c - \sin g_c)}{\cos \theta} \to 0$. Thus, the desired limit follows from $\Psi'_{\theta} = (a_c)'_{\theta} (\sin f_c - \sin g_c) + a_c (\cos f_c \cdot (f_c)' - \cos g_c \cdot (g_c)')$.
\end{proof}

Now we can prove property (B) via the following statement:
\begin{lem}\label{lem:C2}
For any $0 < c < 1$ and $0 < d < 1$, the function $R_{c,d} : \S^2 \to \Re$ is $C^2$-class differentiable.
\end{lem}

\begin{proof}
Clearly, by (\ref{eq:R}), it is sufficient to prove that the function $\rho_c : \S^2 \to \Re$ is $C^2$-class differentiable.
Furthermore, since the function $\rho_c$ is defined by smooth functions of the polar coordinates $\theta$ and $\varphi$ in (\ref{eq:rho}), we need to prove its differentiability properties at the two poles $u_N= (0,0,1), u_S = (0,0,-1) \in \S^2$, associated to the parameter values $\theta = \frac{\pi}{2}$ and $\theta = - \frac{\pi}{2}$ of the polar coordinate system, respectively. We do it only for the North Pole $u_N$, as the proof in the other case is analogous.

Let $U$ be a neighborhood of $u_N$ contained in the open hemisphere defined by the inequality $\theta > 0$ (i.e. the northern hemisphere), and let $\pi_{xy}: U \to H_{xy}$ denote the orthogonal projection onto the $(x,y)$-plane $H_{xy}$. Then $\pi_{xy}$ is a bijection between $U$ and $V=\pi_{xy}(U)$, and for any $u \in U$ with $u=(\cos \theta \cos \varphi, \cos \theta \sin \varphi, \sin \theta)$, the image of $u$ is
\[
\pi_{xy}(u) = (\cos \theta \cos \varphi, \cos \theta \sin \varphi , 0).
\]
For simplicity, we identify the $(x,y)$-plane $H_{xy}$ by $\Re^2$, and ignore the $z$-coordinate of the above point.
Let $h : \Re^2 \to \Re^2$ be defined by $h(\theta, \varphi) = (\cos \theta \cos \varphi, \cos \theta \sin \varphi)$, and note that the Jacobian of $h$ is 
\[
J_h(\theta,\varphi) = \left[ \begin{array}{cc} - \sin \theta \cos \varphi & -\sin \theta \sin \varphi \\ - \cos \theta \sin \varphi & \cos \theta \cos \varphi \end{array} \right],
\]
which is invertible if $\theta \neq \frac{k \pi}{2}$ for some $k \in \mathbb{Z}$, and in this case its inverse is
\begin{equation}\label{eq:JGinv}
J_h^{-1}(\theta,\varphi) =
\left[
\begin{array}{cc}
-\frac{\cos \varphi}{\sin \theta} & -\frac{\sin \varphi}{\cos \theta} \\ -\frac{\sin \varphi}{\sin \theta} & \frac{\cos \varphi}{\cos \theta}
\end{array}
\right]
\end{equation}
Thus, if we reparameterize $U$ with the $(x,y)$-coordinates of the points in $\pi_{xy}(U)$, then, by the chain rule, for any $u \in U \setminus \{ u_N \}$ with
$u(\cos \theta \cos \varphi, \cos \theta \sin \varphi, \sin \theta)$, the partial derivatives of $\rho_c$, as functions of $\theta$ and $\varphi$, can be computed by
\begin{equation}\label{eq:rhodiff}
\left[ \begin{array}{c} (\rho_c)'_x (u) \\ (\rho_c)'_y (u) \end{array} \right] = J_h^{-1}(\theta,\varphi) \left[ \begin{array}{c} (\rho_c)'_{\theta}(u) \\ (\rho_c)'_{\varphi}(u) \end{array} \right] .
\end{equation}
To prove the $C^1$-class differentiability of $\rho_c$ at $u_N$, we show that as $u \to u_N$, we have $(\rho_c)'_x(u), (\rho_c)'_y(u) \to 0$.
Indeed, from (\ref{eq:rhodiff}), we have
\begin{equation}\label{eq:rhodiffexpressed}
(\rho_c)'_x(u) = - \frac{\cos \varphi}{\sin \theta} (\rho_c)'_{\theta}(u) - \frac{\sin \varphi}{\cos \theta} (\rho_c)'_{\varphi}(u),
\end{equation}
and the limit $\lim_{\theta \to \frac{\pi}{2}-0} (\rho_c)'_x(u) = 0$ follows from Lemma~\ref{lem:limits}; the fact that $(\rho_c)'_y(u) \to 0$ as $u \to u_N$ follows similary. On the other hand, by the definition of one-sided partial derivatives, we have
\[
((\rho_c)'_x)_{+}(u_N) = \lim_{\theta \to \frac{\pi}{2}-0} \frac{\rho_c(u)-1}{\cos \theta} = 0
\]
by Lemma~\ref{lem:limits}, where $u=(\cos \theta, 0, \sin \theta)$, and we may obtain similarly that $((\rho_c)'_x)_{-}(u_N)=((\rho_c)'_y)_{+}(u_N)=((\rho_c)'_y)_{-}(u_N)=0$. This yields that $\rho_c$ is $C^1$-class differentiable at $u_N$.

To examine the second partial derivatives of $h$, we can apply the same idea using the partial derivatives $(\rho_c)'_x$ and $(\rho_c)'_y$ playing the role of $\rho_c$.
In particular, for any $u \in U \setminus \{ u_N \}$ with $u=(\cos \theta \cos \varphi, \cos \theta \sin \varphi, \sin \theta)$, we have
\[
\left[ \begin{array}{c} (\rho_c)''_{xx}(u) \\ (\rho_c)''_{xy}(u) \end{array} \right] = J_h^{-1}(\theta,\varphi) \left[ \begin{array}{c} (\rho_c)''_{x\theta}(u) \\ (\rho_c)''_{x \varphi}(u) \end{array} \right],
\]
where $(\rho_c)''_{x\theta}(u)$ and $(\rho_c)''_{x \varphi}(u)$ can be computed, as functions of $\theta$ and $\varphi$, by differentiating the right-hand side of (\ref{eq:rhodiffexpressed}).
From this, we obtain that
\[
(\rho_c)''_{xx}(u) = \left( - \frac{\cos^2 \varphi \cos \theta}{\sin^3 \theta} - \frac{\sin^2 \varphi}{\sin \theta \cos \theta} \right) (\rho_c)'_{\theta}(u) +
\frac{2\cos \varphi \sin \varphi}{\cos^2 \theta} (\rho_c)'_{\varphi}(u) +
\]
\[
+ \frac{\cos^2 \varphi}{\sin^2 \theta} (\rho_c)''_{\theta \theta}(u) + \frac{2\cos \varphi \sin \varphi}{\sin \theta \cos \theta} (\rho_c)''_{\theta \varphi}(u) + \frac{\sin^2 \varphi}{\sin^2 \theta} (\rho_c)''_{\varphi \varphi}(u) .
\]
Using this formula, Lemma~\ref{lem:limits} yields that $\lim_{u \to u_N} (\rho_c)''_{xx}(u) = -1$. On the other hand, using the definition of partial derivatives
and the fact that $(\rho_c)'_x = 0$, if $u=(\cos \theta, 0, \sin \theta)$, then Lemma~\ref{lem:limits} implies that
\[
(((\rho_c)'_x)'_x)_{+}(u_N) = \lim_{\theta \to \frac{\pi}{2}-0} \frac{(\rho_c)'_x(u)}{\cos \theta} = -1,
\]
and $(((\rho_c)'_x)'_x)_{-}(u_N) = -1$ follows similarly. This yields that $(\rho_c)''_{xx}$ exists and is continuous at $u_N$. To prove the same for the functions
$(\rho_c)''_{xy}$ and $(\rho_c)''_{yy}$, we may apply a similar consideration, and obtain, in particular, that $(\rho_c)''_{xy}(u_N)=0$ and $(\rho_c)''_{yy}(u_N)=-1$.
\end{proof}

\begin{rem}
\label{rem:curvature}
Our computation also shows that the surface defined in polar form by $R_{c,d} : \S^2 \to \Re$ has umbilic points at the two poles, and the principal curvatures at $u_N$ and $u_S$ are equal to $\frac{1+2d}{\sqrt{1+d}}$ and $\frac{1-2d}{\sqrt{1-d}}$, respectively. 
\end{rem}

Finally, we prove (C) in Lemma~\ref{lem:equilibria}.

\begin{lem}\label{lem:equilibria}
For any $0 < c,d < 1$, the only equilibrium points of $K_n(c,d)$ with respect to $o$ are $R_{c,d}(u_N) u_N$ and $R_{c,d}(u_S) u_S$.
\end{lem}

\begin{proof}
Note that by the symmetry of $K_n(c,d)$ (and also by the computations in the proof of Lemma~\ref{lem:C2}), the points $R_{c,d}(u_N) u_N$ and $R_{c,d}(u_S) u_S$ are equilibrium points with respect to $o$. To show that no other point satisfies this property it is sufficient to show that if $- \frac{\pi}{2} < \theta < \frac{\pi}{2}$ and $u=(\cos \theta \cos \varphi, \cos \theta \sin \varphi, \sin \theta)$, then $(\rho_c)'_{\theta}(u) \neq 0$ or $(\rho_c)'_{\varphi}(u) \neq 0$.

First, we investigate $(\rho_c)'_{\varphi}(u)$.
Since $\rho_c(u) = a_c (\theta,\varphi) (\sin f_c - \sin g) + \sin g_c$, we have $(\rho_c)'_{\varphi}(u) = (a_c)'_{\varphi}(\theta,\varphi) (\sin f_c - \sin g_c)$.
On the other hand,
\[
(a_c)'_{\varphi} (\theta,\varphi) = \frac{-n \sin \frac{n \varphi}{2} \cos \frac{n\varphi}{2} \cos^2 f_c \cos^2 g_c}{\left( \cos^2 \frac{n\varphi}{2} \cos^2 f_c + \sin^2 \frac{n\varphi}{2} \cos^2 g_c \right)^2},
\]
which is zero only if $\sin \frac{n \varphi}{2} \cos \frac{n\varphi}{2} = \frac{1}{2} \sin (n \varphi) = 0$, implying that $\varphi$ is an integer multiple of $\frac{\pi}{n}$.

Assume that $\varphi$ is an integer multiple of $\frac{\pi}{n}$. This yields that $\rho_c(u) = \sin f_c(u)$ or $\rho_c(u) = \sin g_c(u)$. On the other hand, since $x \mapsto \sin x$, $f_c$ and $g_c$ are strictly increasing on $\left[-\frac{\pi}{2}, \frac{\pi}{2} \right]$ (cf. Lemma~\ref{lem:Fproperties} and the definition of $f_c$ in (\ref{eq:f})), in this case $(\rho_c)'_{\theta}(u) > 0$ follows.
\end{proof}

\subsection{Proving properties (D)-(F) for a suitable element of $\mathcal{F}_n$}\label{subsec:properties}

In this subsection we show that the properties in (D)-(F) are satisfied (simultaneously) for some star-shaped body $K_n(c,d)$ with a suitable choice of $c$ and $d$.
For this purpose, we set some arbitrary positive constant $\varepsilon > 0$, and observe that the definition of $\rho_c$ implies that $\max \{ \rho_c(u) : u \in \S^2 \} = 1$, and hence, if $0 < d < \varepsilon$, then $(1-d) \BB^3 \subseteq K_n(c,d) \subseteq  (1+d) \BB^3$ and $d_H(K_n(c,d),\BB^3) < \varepsilon$.

First, in Lemma~\ref{lem:centering} we examine (E). 

\begin{lem}\label{lem:centering}
There exist constants $0 < c_1 < c_2 < 1$, $0 < d_0 < 1$ and a function $F:(0,d_0) \to [c_1,c_2]$ such that for any $d \in (0,d_0)$.
the center of mass of $K_n(F(d),d)$ is $o$.
\end{lem}

\begin{proof}
We start the proof with the observation that as the symmetry group of $K_n(c,d)$ is $D_n$ for some $n \geq 2$, the center of mass of $K$ lies on the $z$-axis.
Thus, the center of mass of $K_n(c,d)$ is $o$ if and only if the first moment of $K_n(c,d)$ with respect to the $(x,y)$-plane is $0$.

Changing to spherical coordinates, the first moment of $K_n(c,d)$ with respect to the $(x,y)$-plane can be computed as
\begin{equation}\label{eq:firstmoment}
M_{xy}(K_n(c,d)) = \int_{v \in K_n(c,d)} z \, d v = \int_{-\pi/2}^{\pi/2} \int_0^{2\pi} \frac{ \left( 1+d \rho_c \left(\theta, \varphi \right)\right)^4 }{4} \cos \theta \sin \theta \, d \varphi \, d \theta 
\end{equation}

Expanding the integrand, we obtain that the above moment is $M_{xy}(K_n(c,d)) = d H(c,d)$, where
\begin{equation}\label{eq:H}
H(c,d)=\int_{-\pi/2}^{\pi/2}\int_0^{2\pi}  \left( \rho_c(\theta,\varphi)+ \frac{3}{2} d \rho_c^2(\theta,\varphi) + d^2 \rho_c^3(\theta, \varphi) + \frac{d^3}{4} 
\rho_c^4(\theta, \varphi) \right) \sin \theta \cos \theta \, d \varphi \, d \theta .
\end{equation}
Here, since the integrand is continuous on the domain for every $0 < c \leq 1$ and $0 \leq d \leq 1$, $H(c,d)$ is also continuous in this region.

Recall that $f_1(\theta)= g_1(\theta) = \theta$ by Lemma~\ref{lem:Fproperties} and Equation (\ref{eq:f}). Thus, $\rho_1(\theta, \varphi)=\sin \theta$, and for $c=1$ we have
\[
H(1,d) = 2 \pi \int_{-\pi/2}^{\pi/2} (\left( \sin^2 \theta + \frac{3}{2} d \sin^3 \theta + d^2 \sin^4 \theta + \frac{d^3}{4} \sin^5 \theta \right)
\cos \theta \, d \theta = 
\]
\[
= 2 \pi \int_{-1}^1 x^2 + \frac{3}{2} d x^3 + d^2 x^4 + \frac{1}{4} d^3 x^5 \, dx = \frac{4\pi}{5}d^2 + \frac{4\pi}{3},
\]
which is strictly positive if $0 \leq d \leq 1$ (see Fig. \ref{fig:rho0contour}). Thus, by compactness, there is a value $0 < c_2 < 1$, arbitrarily close to $1$, such that
$M_{xy}(K_n(c_2,d)) > 0$ for all $d \in (0,1]$.

Next, we examine the case $d=0$, and we show that if $c > 0$ is sufficiently small, then $H(c,0) < 0$ 
(see also Fig.~\ref{fig:rho0contour}). 
Note that we have
\[
H(c,0) =\int_{-\pi/2}^{\pi/2}\int_0^{2\pi}  \rho_c(\theta,\varphi) \cos \theta \sin \theta  \, d \varphi \, d \theta .
\]
For any $-\frac{\pi}{2} < \theta < \frac{\pi}{2}$, an elementary computation, using the well-known limit $\lim_{x \to 0} \frac{\sin x}{x} = 1$ and (\ref{eq:a}), yields that if $\varphi \neq \frac{(1+2k)\pi}{n}$ for some $k \in \mathbb{Z}$, then
\[
\lim_{c \to 0^+} a_c(\theta,\varphi) = \lim_{c \to 0^+} \frac{1}{1 + \tan^2 \frac{n\varphi}{2} \frac{\cos^2 (g_c(\theta))}{\cos^2 (f_c(\theta))}}  = \frac{\left( \frac{\pi}{2} + \theta\right)^4 \cos^2 \frac{n \varphi}{2}}{\left( \frac{\pi}{2} + \theta\right)^4 \cos^2 \frac{n \varphi}{2} + \left( \frac{\pi}{2} - \theta\right)^4 \sin^2 \frac{n \varphi}{2} },
\]
and a similar computation for $\lim_{c \to 0^+} (1-a_c(\theta,\varphi))$ yields the same limit if $\varphi \neq \frac{2k\pi}{n}$ for some $k \in \mathbb{Z}$. Thus, we have
\[
\lim_{c \to 0^+} a_c(\theta,\varphi) = \frac{\left( \frac{\pi}{2} + \theta\right)^4 \cos^2 \frac{n \varphi}{2}}{\left( \frac{\pi}{2} + \theta\right)^4 \cos^2 \frac{n \varphi}{2} + \left( \frac{\pi}{2} - \theta\right)^4 \sin^2 \frac{n \varphi}{2} }
\]
for any $-\frac{\pi}{2} < \theta < \frac{\pi}{2}$.
This implies that
$\lim_{c \to 0^+} r(\varphi,\theta,c) = (1-2 \lim_{c \to 0^+} a(\varphi,\theta,c)) =
\frac{\left( \frac{\pi}{2} - \theta\right)^4 \sin^2 \frac{n \varphi}{2} - \left( \frac{\pi}{2} + \theta\right)^4 \cos^2 \frac{n \varphi}{2}}{\left( \frac{\pi}{2} + \theta\right)^4 \cos^2 \frac{n \varphi}{2} + \left( \frac{\pi}{2} - \theta\right)^4 \sin^2 \frac{n \varphi}{2} }$ for any $-\frac{\pi}{2} < \theta < \frac{\pi}{2}$. We define this quantity as $\rho_0(\theta,\varphi)$.
Note that $|\rho_c(\theta,\varphi)| \leq 1$ for all values $0 < c \leq 1$, $0 \leq  \varphi \leq 2\pi$, $-\frac{\pi}{2} < \theta < \frac{\pi}{2}$, and the constant function $(\theta,\varphi) \mapsto 1$ is integrable on $[0,2\pi] \times \left[-\frac{\pi}{2}, \frac{\pi}{2} \right]$.
Thus, by Lebesgue's dominated convergence theorem \cite{Evans}, we have
\[
\lim_{c \to 0^+} H(c,0) =\int_{-\pi/2}^{\pi/2}\int_0^{2\pi} \rho_0(\theta,\varphi) \cos \theta \sin \theta  \, d \varphi \, d \theta = -2.3168\ldots,
\]
which is negative. This yields that there are some values $0 < c_1 < c_2$ and $0 < d_0$ such that for any $0 \leq d \leq d_0$, $H(c_1,d) < 0$, and the assertion follows from the continuity of $H$. 
\end{proof}

Our next lemma takes care of Property (D). Recall \cite{Schneider} that $\mathcal{C}^2_+$ denotes the family of convex bodies whose boundary is $C^2$-class differentiable at every point, and has strictly positive Gaussian curvature everywhere.

\begin{lem}\label{lem:convexity}
Let $[c_1,c_2]$ be defined as in Lemma~\ref{lem:centering}. Then there is a value $0 < d^{\star} < 1$ such that for any 
$0 < d < d^{\star}$ and $c \in [c_1,c_2]$, we have $K_n(c,d) \in \mathcal{C}^2_+$. 
\end{lem}

\begin{proof}
First, we remark that by Lemma~\ref{lem:C2}, it is sufficient to prove that under the conditions in the lemma, the principal curvatures at any point of the surface defined by $R_{c,d} : \S^2 \to \Re$ are strictly positive.

Note that by Lemma~\ref{lem:C2}, $R_{c,d}(u)$ is a $C^2$-class function of $u$ and for any fixed $u \in \S^2$, it is a continuous function of $c$ and $d$ with $(c,d) \in [c_1,c_2] \times [0,1]$. Furthermore, as $d \to 0+0$, the principal curvatures of the surface $S(c,d) = \{ R_{c,d}(u) u, u \in \S^2\}$ defined by this function tend to $1$. Thus, for any $(u,c) \in \S^2 \times [c_1,c_2]$ there is some $\delta(u,c)$ such that for any $(u',c',d') \in \S^2 \times [c_1,c_2] \times [0,1]$, if $||u'-u|| < \delta(u,c)$, $|c'-c| < \delta(u,c)$ and $0 \leq d' < \delta(u,c)$, then the principal curvatures of $S(c',d')$ at $R_{c',d'}(u') u'$ are strictly positive.

Let us call the set of points $(u',c') \in \S^2 \times [c_1,c_2]$ satisfying the inequalities 
$||u'-u|| < \delta$ and $|c'-c| < \delta$ the \emph{$\delta$-neighborhood} of $(u,c)$, and note that this set is open in $\S^2 \times [c_1,c_2]$.
Hence, by the compactness of $\S^2 \times [c_1,c_2]$, there are finitely many points $(u^i,c^i) \in \S^2 \times [c_1,c_2]$, where $i=1,2,\ldots, m$, such that the $\delta(u^i,c^i)$-neighborhoods of $(u^i,c^i)$ cover $\S^2 \times [c_1,c_2]$. Let $d^{\star} = \min \{ \delta(u^i,c^i) : i=1,2,\ldots, m \}$. Then, for any
$c \in [c_1,c_2]$ and $0 < d < d^{\star}$, the principal curvatures of the boundary of $K_n(c,d)$ at any point are strictly positive, implying that $K_n(c,d) \in \mathcal{C}^2_+$.
\end{proof}

We note that numerical solution of the equation $\lim_{d \to 0^+} H(c_0,d) = 0$ for $n=3$ yields $c_0=0.056...$, furthermore $K_3(c_0,d)$ is convex if $d<d^\star=0.0013...$. Due to the fact that $d^\star<<1$, the set $K_3(c_0,d^\star)$ is an almost perfect sphere which is illustrated with level sets in Fig.~\ref{fig:zebra}.

\begin{figure}
\centering
\includegraphics[width=10cm]{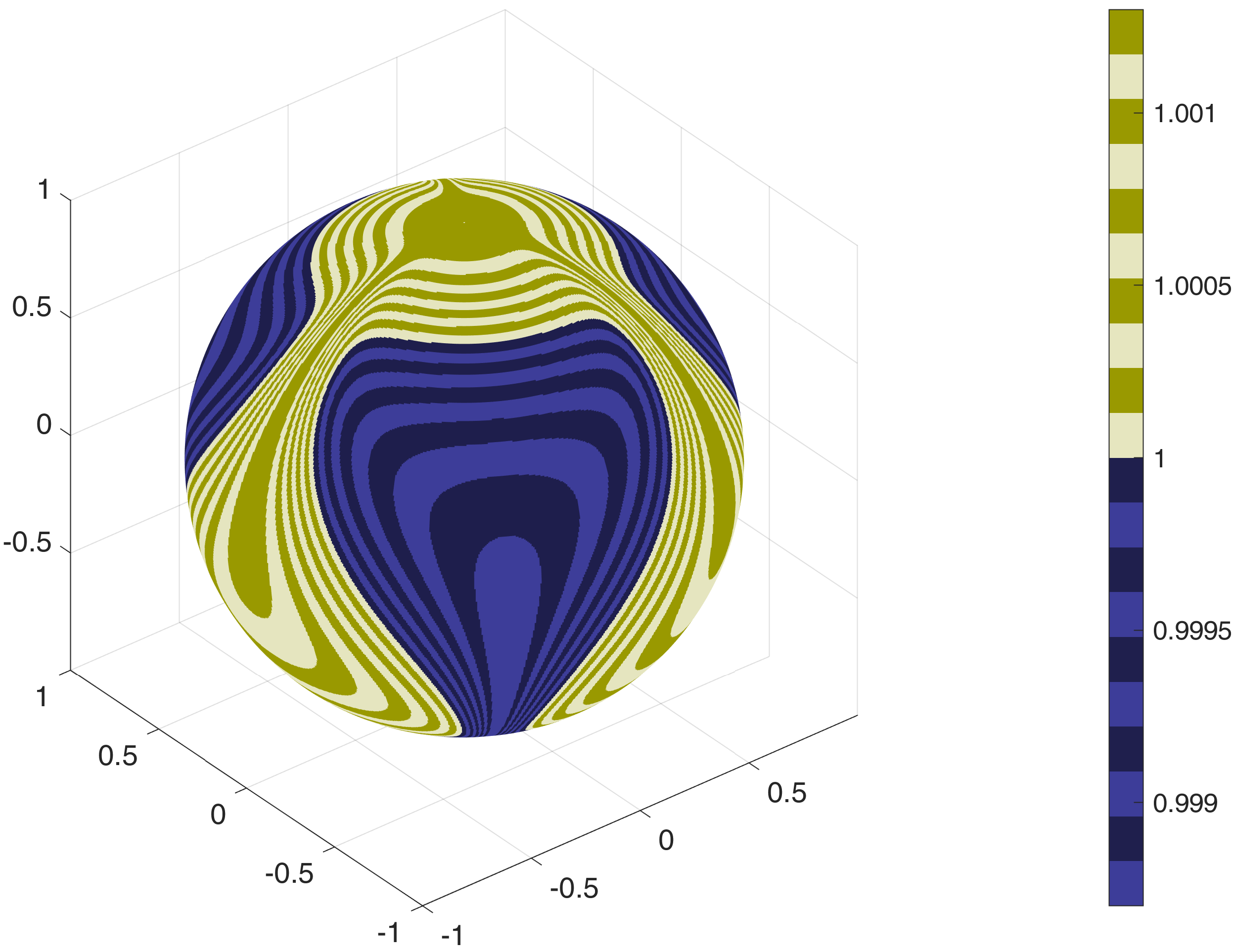}

\caption{Illustration of the homogeneous, convex, mono-monostatic body $K_3(0.056,0.0013)$ with $D_3$-symmetry. The shape is very close to a sphere. To indicate the deviations,  color scale is defined by the value of $R$, regions bounded by some chosen $R=constant$ level sets appear in constant color. Compare also with Fig.~\ref{fig:rho0contour}(b).}
\label{fig:zebra}
\end{figure}

\section{The proof of Theorem~\ref{thm:main}}\label{sec:proofofmain}

\subsection{The proof of (i) of Theorem~\ref{thm:main}}\label{subsec:i}

Let $K$ be a nondegenerate, mono-monostatic convex body with symmetry group $G$, and assume that the center of mass of $K$ is $o$.
First, we recall Lemma 1 from \cite{langi}, and remark that the fact that there is no rotationally symmetric nondegenerate convex body can be found in \cite{DK} in more detail.

\begin{lem}\label{lem:finite}
The symmetry group of any nondegenerate convex body $K$ is finite.
\end{lem}

By Lemma~\ref{lem:finite}, $G$ is a discrete subgroup of $\mathcal{O}(3)$. On the other hand, since $K$ has a unique stable and unstable point, every element of $G$ leaves also these points fixed. This implies that there is a $1$-dimensional linear subspace of $\Re^3$ fixed under every element of $G$. Thus, $G$ is isomorphic to a $2$-dimensional point group.

\subsection{The proof of (ii) of Theorem~\ref{thm:main}}\label{subsec:ii}

First, we investigate the case that $\mathcal{F} = (1,1)_c$.
In Section~\ref{sec:construction}, we have shown that for any $n \geq 2$, there is an element $K$ of $(1,1)_c$ whose symmetry group is $D_n$.
From this, (ii) for $(1,1)_c$ follows from the next two simple observations.

Before stating Lemma~\ref{lem:C1}, we recall that a convex body $K$ in $\Re^d$ is called \emph{smooth} if for every boundary point of $K$ there is a unique hyperplane supporting $K$ at $p$, or equivalently, the boundary of $K$ is $C^1$-class differentiable (cf. \cite{Schneider}).

\begin{lem}\label{lem:C1}
If $K_1,K_2 \subset \Re^d$ are smooth convex bodies, then so is $\conv (K_1 \cup K_2)$.
\end{lem}

\begin{proof}
Clearly, $K=\conv (K_1 \cup K_2)$ is compact, convex, and it has nonempty interior, i.e. it is a convex body. Let $p \in \bd (K)$. We show that there is a unique hyperplane supporting $K$ at $p$. If $p \in K_i$ for some $i=1,2$, then this statement follows from the facts that $K_i$ is smooth, and any hyperplane supporting $K$ at $p$ supports also $K_i$. Thus, assume that $p \in \bd (K) \setminus(K_1 \cup K_2)$. Then there are distinct points $p_1 \in K_1$ and $p_2 \in K_2$ and some $t \in (0,1)$ such that $p = tp_1 + (1-t)p_2$. Let $L$ be the straight line through $p_1$ and $p_2$. Observe that any hyperplane supporting $K$ at $p$ contains $L$. Hence, if $H_1$ and $H_2$ are distinct hyperplanes supporting $K$ at $p$, then they support $K_i$ at $p_i$ for all $i=1,2$, which contradicts the condition that $K_1$ and $K_2$ are smooth.
\end{proof}

\begin{prop}\label{prop:subgroup}
Let $K$ be an element of $(S,U)_c$ for some $S,U \geq 1$, and let $G$ denote the symmetry group of $K$. Then for every $\varepsilon > 0$ and every subgroup $H$ of $G$, there is an element $K'$ of $(S,U)_c$ such that $d_H(K,K') < \varepsilon$, and the symmetry group of $K'$ is $H$.
\end{prop}

\begin{proof}
Without loss of generality, assume that the center of mass of $K$ is $o$, and let the equilibrium points of $K$ be $p_1, \ldots, p_N$, with $N=2S+2U-2$. For any $i=1,2,\ldots, N$, let $U_i$ be an open neighborhood of $p_i$ in $\bd(K)$ such that the eigenvalues of Hessian of the Euclidean distance function from $o$ do not change signs in $V_i$. By the nondegeneracy of $K$, $o$ has an open neighborhood $U$ such that $K$ has the same number of equilibrium points with respect to any point $c \in U$, and each $V_i$ contains exactly one equilibrium point with respect to any $c \in U$, and the type of this point coincides with the type of $p_i$.

Let $B$ be a closed ball, and let $K_B = \conv (K \cup \bigcup_{\sigma \in H} \sigma (B))$. Clearly, by Lemma~\ref{lem:C1}, $K_B$ is a smooth convex body for all choices of $B$. On the other hand, if the radius of $B$ is sufficiently small and its center is sufficiently close to a boundary point of $K$ outside the closures of all the $V_i$s, then $\bd(K_B) \cap \bd(K)$ contains all $V_i$s, and the center of mass of $K_B$ is contained in $U$. Let $K'$ be chosen as such a convex body $K_B$.

Then $K'$ has $S$ stable and $U$ unstable points, each equilibrium point of $K'$ has a $C^2$-class neighborhood in $\bd(K')$, and $\bd(K')$ has strictly positive principal curvatures at any equilibrium point. In other words, in this case $K' \in (S,U)_c$. As the symmetry group of $K'$ is clearly $H$, $K'$ satisfies the conditions in the lemma.
\end{proof}

Until now, we have proved (ii) of Theorem~\ref{thm:main} for $\mathcal{F} = (1,1)_c$. The proof for $\mathcal{F} = (1,1)_p$ readily follows from Theorem 2 of \cite{langi}, stated as follows.

\begin{thm}\label{thm:approx}
Let $\varepsilon > 0$, $S,U \geq 1$ be arbitrary, and let $G$ be any subgroup of the orthogonal group $\mathcal{O}(3)$. Then for any centered, $G$-invariant convex body $K \in (S,U)_c$, there is a centered $G$-invariant convex polyhedron $P \in (S,U)_p$ such that $d_H(K,P) < \varepsilon$.
\end{thm}

\section{Additional remarks}

\begin{rem}                                                                                                     
In the proof of Theorem~\ref{thm:main} we have shown that for any discrete $2$-dimensional point group $G$ and any $\varepsilon > 0$, there is an element $K \in \mathcal{F}$ whose symmetry group is $G$ and which satisfies $d_H(K,\BB^3) < \varepsilon$. Nevertheless, Lemma 4 of \cite{dedicata}, and the observation that the smoothing subroutine described in its proof preserves the symmetry group of the body, yield that the statement in (ii) of Theorem~\ref{thm:main} for $\mathcal{F}=(1,1)_c$ holds under the additional requirement that the body $K$ has a $C^{\infty}$-class differentiable boundary.
\end{rem}

\begin{rem}
Clearly, the proof of (i) of Theorem~\ref{thm:main} in Subsection~\ref{subsec:i} holds for convex bodies with a single stable \emph{or} a single unstable point, implying that for any $S,U \geq 1$, the symmetry group of any representative of any of the classes $(S,1)_c$, $(S,1)_p$, $(1,U)_c$, $(1,U)_p$ is a discrete $2$-dimensional point group.
\end{rem}

Finally, we recall a related open question from \cite{DKLRV}: The \emph{mechanical complexity} of a nondegenerate convex polyhedron $P$ is $C(P)=n(P)-N(P)$, where $N(P)$ denotes the total number of equilibrium points of $P$, and $n(P)$ denotes the total number of vertices, edges and faces of $P$, and furthermore, the mechanical complexity $C(S,U)$ of a class $(S,U)_p$ is defined as the minimum mechanical complexity of the elements of $(S,U)_p$. Here, despite the fact that the existence of a mono-monostatic convex polyhedron was shown in \cite{langi} (and follows also from Theorem~\ref{thm:main}), presently we have no explicit upper bound on the mechanical complexity of the class $(1,1)_p$. To motivate research in this direction, a prize has been offered in \cite{DKLRV} establishing the mechanical complexity $C(1,1)$, the amount $p$ of which is given in US dollars as
\begin{equation}
p=\frac{10^6}{C(1,1)}.
\end{equation}
For more details of this prize, the interested reader is directed to \cite{DKLRV}.

\end{document}